\documentclass[12pt]{amsart}
\usepackage[utf8x]{inputenc}

\usepackage{dsfont}
\usepackage{amsmath}
\usepackage{amssymb}
\usepackage{graphicx}
\usepackage{amssymb}
\usepackage{amsfonts}
\usepackage{soul}
\usepackage{mathpazo}
\usepackage{color}
\usepackage{yfonts}
\usepackage{paralist}
\usepackage{stmaryrd}
\usepackage{amsxtra}
\usepackage{latexsym,mathrsfs}
\usepackage{verbatim}
\usepackage{mathabx}

\usepackage{epsfig, enumerate}
\usepackage{color}
\usepackage{hyperref}

\newtheorem{theorem}{Theorem}
\newtheorem*{lemma*}{Claim}
\newtheorem{lemma}[theorem]{Lemma}

\newtheorem{conjecture}[theorem]{Conjecture}
\newtheorem{prop}[theorem]{Proposition}

 \theoremstyle{definition}\newtheorem{definition}[theorem]{Definition}

\numberwithin{theorem}{section}
\numberwithin{equation}{section}
\definecolor{turquoise}{cmyk}{0.65,0,0.1,0.1}
\definecolor{purple}{rgb}{0.65,0,0.65}
\definecolor{green}{rgb}{0, 0.5, 0}
\definecolor{blue}{rgb}{0, 0, 1}
\definecolor{orange}{rgb}{0.8, 0.6, 0.2}
\definecolor{red}{rgb}{0.8, 0.2, 0.2}
\definecolor{brown}{rgb}{0.5, 0.16, 0.16}

\title[Tur\'an number]{Tur\'an Numbers of Subdivisions of Multipartite Graphs}

\author{Xiao-Chuan Liu}
\address[Liu]{Departamento de Matemática,
 Universidade Federal de Pernambuco,
	Avenida Jornalista Aníbal Fernandes - Cidade Universitária, Recife, Brasil}
\email{xiaochuan.liu@ufpe.br}

\author{Danni Peng}
\address[Peng]{Instituto de Matemática Pura e Aplicada,
Estrada Dona Castorina 110, Jardim Botânico, Rio de Janeiro, 22460-320, Brasil}
\email{dannipeng49@gmail.com}

\author{Xu Yang}
\address[Yang]{Institute of Computing, Federal University of Alagoas,
	Av. Lourival Melo Mota, S/N, Maceió, Brasil}
\email{yang@ic.ufal.br}

\begin{document}
\maketitle{}

\begin{abstract} In this paper, we investigate the Tur\'an exponent for $1$-subdivisions of graphs that are neither bipartite nor complete. Specifically,
we establish an upper bound on the Turán number of the 1-subdivision of $K_{s,t}^+$, where $K_{s,t}^+$ is obtained by adding a single edge within the part of size $s$ of the complete bipartite graph $K_{s,t}$, with $4\leq s \leq t$. In addition, we derive an upper bound for the extremal number of a family of graphs formed by (possibly degenerate) 1-subdivisions of certain tripartite graphs.


\end{abstract}

\section{Introduction}
One of the most fundamental problems in extremal graph theory is to determine the
Tur\'an number of a graph $H$.
By definition, let $\text{ex}(n,H)$ denote the maximum possible number of edges of an $H$-free graph on $n$ vertices.
The cornerstone in this area is the classical
Erd{\H{o}}s-Stone theorem~\cite{erdos1946structure}, which says that
\begin{equation}
\text{ex}(n,H)= \bigg(1-\frac{1}{\chi(H)-1}+o(1) \bigg) {n \choose 2},
\end{equation}
where $\chi(H)$ is the chromatic number of $H$.
However, when $H$ is a bipartite graph, the above theorem only gives the crude bound $o(n^2)$. In fact,
to understand the behavior of the extremal numbers of bipartite graphs
is one of the central topics in extremal combinatorics.
There are several general conjectures (see ~\cite{Furedi2013history}).
One of them is the Rational Exponents Conjecture
 proposed by Erd{\H{o}}s and Simonovits ~\cite{erdHos1981combinatorial}, which states that for every rational $r\in (1,2)$, there exists a graph $H$ such that $\text{ex}(n,H)=\Theta(n^r)$.
 Bukh and Conlon ~\cite{bukh2018rational}  took a major step towards the conjecture by showing that there exists a family of graphs $\mathcal{H}$ such that $\text{ex}(n,\mathcal{H})=\Theta(n^r)$ for any rational $r\in(1,2)$.
 An alternative approach was suggested by Kang, Kim and Liu~\cite{kang2021rational}, who introduced the following "Subdivision Conjecture" and showed that it implies the Rational Exponent Conjecture. For any graph $H$, its 1-subdivision $H^{\text{sub}}$ is obtained by replacing each
edge in $E(H)$ by a path of length two.
In other words, $H^{\text{sub}}$ is a bipartite graph with a vertex bipartition represented by $V(H) \cup E(H)$, and its edge set is defined via the incidence relations of $H$.
The Subdivision Conjecture is as follows.

\begin{conjecture}[Kang, Kim and Liu~\cite{kang2021rational}]
For any bipartite graph $H$, if $\text{ex}(n,H)=O(n^{1+\alpha})$ for some $\alpha>0$, then $\text{ex}(n,H^{\text{sub}})=O(n^{1+\frac \alpha 2})$.
\end{conjecture}


The conjecture holds for all non-bipartite graphs $H$, as a well-known result of Furedi~\cite{Fredi1991OnAT} states that
$\text{ex}(n,H^{\text{sub}})=O(n^{3/2})$ (a simple proof using dependent random choice was given later by Alon, Krivelevich and Sudakov~\cite{alon2003turan}).
Moreover, resolving a conjecture made by
Erd{\H{o}}s~\cite{erdous1988problems},
Conlon and Lee~\cite{conlon2021extremal} showed that the Tur\'an exponent of the 1-subdivision of the clique $K_t$ is strictly less than $3/2$.
Not long afterwards, Janzer~\cite{janzer2018improved} improved this to
\begin{equation}\label{complete_janzer}
	\textup{ex}(n,K_t^{\text{sub}})=O(n^{\frac32-\frac1{4t-6}}),
\end{equation}
which is sharp in the case $t=3$ since $\textup{ex}(n,C_6)=\Theta(n^{4/3})$(see~\cite{bondy1974cycles}).



Conlon and Lee~\cite{conlon2021extremal} also proved a similar result for the subdivision of the bipartite graph $K_{s,t}$, showing that the Tur\'an exponent of the subdivision of $K_{s,t}$ is not greater than $3/2-1/12s$.
In particular, since the Tur\'an exponent of $K_{s,t}$ is $2-1/s$ if $t$ is sufficiently large (see~\cite{kollar1996norm,kHovari1954problem,bukh2021extremal}),
Subdivision Conjecture suggests that  $3/2-1/2s$ should be the correct exponent of $K_{s,t}^{\text{sub}}$.
After a few months,
Conlon, Janzer and Lee~\cite{conlon2021more} obtained the following result.

\begin{theorem}[Conlon, Janzer and Lee~\cite{conlon2021more}]\label{more}
For any integers $2 \leq s\leq t$, $$\textup{ex}(n,K_{s,t}^{\text{sub}})=O(n^{\frac32-\frac{1}{2s}}).$$
\end{theorem}
When $t$ is sufficiently large,
one can combine the recently developed random polynomial method of Bukh and Conlon~\cite{bukh2015random}
to show that the exponent indeed provide the correct order (see~\cite{bukh2015random,  bukh2021extremal, bukh2018rational, Conlon2019graphs}).
 More precisely, for any integer $s$ and for any $t$ sufficiently large depending on $s$,
\begin{equation}\label{matchingKst}
	\textup{ex}(n, K_{s,t}^{\text{sub}})= \Theta(n^{\frac32-\frac1{2s}}).
\end{equation}

In recent years, there has been a great deal of interest in bounding the extremal numbers of subdivisions of other bipartite graphs, see for example \cite{janzer2020extremal,janzer2021extremal,jiang2020turan,jiang2023many,sudakov2020turan}. Nevertheless, a proof of the Subdivision Conjecture still seems far away, and it is therefore natural to study the 1-subdivision of non-bipartite graphs in order to develop a more comprehensive theory.
Janzer~\cite{janzer2018improved} obtained some bounds for a particular class of non-bipartite graphs (see below), but the current situation is still very unsatisfactory, and there is still no conjecture for non-bipartite $H$ analogous to the one introduced by Kang, Kim and Liu~\cite{kang2021rational}.

In other words, in general, if the chromatic number of $H$ is $k\geq 3$, although its Tur\'an number is of order $\Theta(n^2)$, there is no theoretical prediction available for the Tur\'an number of its $1$-subdivision graph $H^{\text{sub}}$, with a single exceptional case of the complete graph $K_k$, for which~(\ref{complete_janzer}) provides a potentially sharp upper bound. Here we contribute to this topic by studying  the extremal number of the $1$-subdivisions of non-bipartite (and non-complete) graphs.

%
%
%

In this paper, we  study the extremal number of the 1-subdivision of the tripartite graph $K_{s,t}^+$, which is obtained from the complete bipartite graph $K_{s,t}$ by adding a single edge inside the part of size $s$.

\begin{theorem}\label{thm_2}
For any positive integers $4 \leq s \leq t$,
\begin{equation}
\textup{ex}(n, (K_{s,t}^+)^{\text{sub}})= O(n^{\frac 32 -\frac{s-2}{2(s-1)s}}).
\end{equation}
\end{theorem}
The best known lower  bound of the Tur\'an number of $ (K_{s,t}^+)^{\text{sub}}$ is from the lower bound of $\text{ex}(n, K_{s,t}^{\text{sub}})$, whose exponent is $\frac32-\frac{1}{2s}$ if $t$ is sufficiently large, as shown in ~\eqref{matchingKst}.   The difference of exponents of the known lower bound and our upper bound is $\frac1{2(s-1)s}$.

We also establish upper bounds on the exponents for the subdivisions of
a family of graphs $\mathcal{F}_{1,s,t}$, which includes $(K_{1,s,t})^{\text{sub}}$.
Let $\mathcal{F}_{1,s,t}$ denote the family of bipartite graphs constructed as follows: 

We start with a copy of $K_{s,t}^{\text{sub}}$, the subdivision of the complete bipartite graph $K_{s,t}$. Recall that the vertex set of $K_{s,t}^{\text{sub}}$ can be viewed  as  $V(K_{s,t})\cup E(K_{s,t})$, where each edge of $K_{s,t}$ corresponds to a new subdivision vertex. Next, we introduce a new vertex $r$.
For each vertex $v\in V(K_{s,t})$, we attach a path of length 2 from $r$ to  $v$.
The internal (middle) vertices of each such path can be chosen in one of the following two ways:
\begin{enumerate}
\item From the set $E(K_{s,t})$, i.e., the subdivision vertices of $K_{s,t}^{\text{sub}}$; 
\item From a new set $N$ of vertices, where each element of $N$ is distinct from $V(K_{s,t}^{\text{sub}})\cup \{r\}$.
\end{enumerate}
Importantly, the internal vertices of these $s+t$ paths may be repeated.
The family $\mathcal{F}_{1,s,t}$ consists of all bipartite graphs that can be obtained through this construction.

Should we choose to create these new $s+t$ paths
using pairwise distinct vertices, we obtain the graph $K_{1,s,t}^{\text{sub}}$.
Since $K_{1,s,t}^{\text{sub}}$ serves as our primary object of interest, we refer to the remaining graphs in this family as {\it degenerate}. In Figure \ref{figure}, we illustrate two examples: the first is $K_{1,2,3}^{sub}$ and the second is a degenerate case.
Inside each dashed rectangle, we display a  copy of $K_{2,3}^{\text{sub}}$, whose vertex set is the union of $V(K_{2,3})=\{y_1,y_2,z_1,z_2,z_3\}$ and $E(K_{2,3})=\{b_{11},b_{12},b_{13},b_{21},b_{22},b_{23}\}$.

\begin{figure}[h]
	\centering
	\includegraphics[width=1\textwidth]{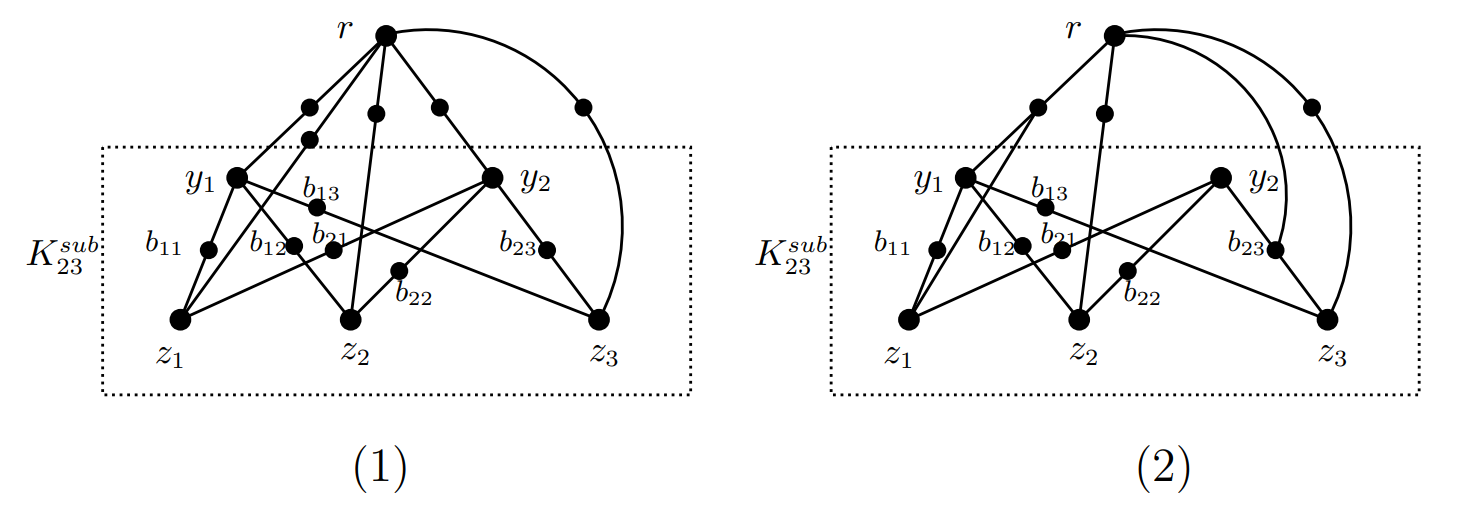}
	\caption{(1) is $K_{1,2,3}^{\text{sub}}$, a non-degenerate example. (2) is a degenerate example. Here $ry_1$ is connected with a new vertex $Y$. $rz_1$ used the same vertex $Y$. $ry_2$ is connected with $b_{23}$. Finally, $rz_2$ and $rz_3$ are both connected with the corresponding new vertices. }\label{figure}
\end{figure}

\begin{theorem}\label{family_version} For any integers $3\leq s \leq t$,
 $\text{ex}(n,  \mathcal F_{1,s,t})= O(n^{1+\frac{s^2+1}{2s+2s^2}})$.
\end{theorem}

Note that for all $s\geq 3$,  $\frac{s}{2s+2} < \frac{s^2+1}{2s+2s^2} <\frac{s}{2s+1} $,
and $\frac{s}{2s+1}$ was the best previous upper bound for this problem,
given by Janzer~\cite{janzer2018improved}.
Our result improves this upper bound.
We suspect the same upper bound is in fact true for $K_{1,s,t}^{\text{sub}}$,
which leads to the following conjecture.

\begin{conjecture}
For any integers $3\leq s \leq t$,
	\begin{equation}
\text{ex}(n,K_{1,s,t}^{\text{sub}})=O(n^{1+\frac{s^2+1}{2s+2s^2}}).
		\end{equation}
\end{conjecture}

In order to achieve this upper bound,
more ideas are needed. The real difficulty is to avoid the {\it degeneracy} when we try to embed the 1-subdivision of a non-bipartite non-complete graph.
By saying non-degeneracy, we mean that during the entire embedding scheme, we have to make sure the new common neighbor we want to choose at each step must not be already used before. This causes a lot of trouble when the embedding scheme proceeds (see Figure 1).
In Theorem ~\ref{thm_2}, we were able to overcome exactly this difficulty via some sophisticated counting arguments.



 The rest of the paper is organized as follows. First, in Section 2, we will recall (and slightly reformulate) Janzer’s method for estimating extremal numbers of subdivisions.
We found this version slightly easier to use than the original formulation, and believe that it could provide a useful alternative perspective on the method.
	 In Section 3, we prove Theorem~\ref{thm_2}. 
In Section 4, we prove Theorem~\ref{family_version}.

We use the notation $f(n) = \omega(g(n))$
 to refer to the condition that
  for any constant $C > 0$, there exists an integer $n_0$ such that $f(n) \geq Cg(n)$ for any $n\geq n_0$.
Similarly, $f(n) = o(g(n))$ denotes that for any constant $\varepsilon > 0$, there exists an integer $n_0$ such that $f(n) \leq \varepsilon g(n)$ for any $n\geq n_0$.
Hereafter, we assume that $n$ is a sufficiently large integer.

\section{Preliminary Considerations}
In this section, we recall some methods for embedding  $1$-subdivisions that were introduced by Janzer~\cite{janzer2018improved}, and by Conlon, Janzer and Lee~\cite{conlon2021more}.
We begin by recalling some simple reduction lemmas that we will find useful when discussing these methods.
\subsection{Standard Reductions}
In this subsection, we recall two standard lemmas, and a useful reduction lemma from ~\cite{conlon2021extremal}. The following lemma allows us to assume that our host graph is bipartite and has large minimum degree.

\begin{lemma}\label{lowerbound_degree_a}
Let $G$ be a graph with average degree $d_G=2e(G)/v(G)$. Then there exists a subgraph of $G$ which is bipartite and every vertex has degree at least $d_G/4$.
\end{lemma}
Recall also that $\textup{ex}(n,T)\leq kn$ for every tree $T$ with $k$ vertices. We will use the following lemma to control the minimum degree in each part of an unbalanced bipartite graph. Let $\delta_X(G)$ be the minimal degree
in G among vertices in a set $X\subset V(G)$.

\begin{lemma}\label{lowerbound_degree_b}
 Let $G$ be a bipartite graph with vertex bipartition $X\cup Y$, write $d_X= e(G)/ |X|$ and $d_Y= e(G)/|Y|$. Then there exists a subgraph $G'$ of $G$, on some bipartition $X'\cup Y'$, where $X'\subset X$ and $Y'\subset Y$, so that
  $$\delta_{X'}(G')\geq d_X/4,\quad \delta_{Y'}(G')\geq d_Y/4 \quad \text{and} \quad e(G')\geq e(G)/2.$$
\end{lemma}
	

A graph $G$ is called \textit{$K$-almost regular},
if every vertex $v$ has degree in some interval of the form
$[d, Kd]$.
A bipartite graph with bipartition $X\cup Y$ is called
\textit{balanced},  if $1/2\le |X|/|Y|\le 2$.
The following useful lemma will allow us to assume that our host graph is balanced and K-almost regular.

\begin{lemma}[Lemma 2.3 of~\cite{conlon2021extremal}]\label{make_assumption}
Fix $0<\alpha< 1$, and let $G$ be a graph with $n$ vertices and $e(G)= \omega( n^{1+\alpha})$.
Then $G$ has a $K$-almost regular balanced bipartite subgraph $G'$ with $e(G') =\omega( v(G')^{1+\alpha})$.
\end{lemma}

Note that the condition $e(G') =\omega( v(G')^{1+\alpha})$ implies that
$v(G') \to \infty$ as $n\to \infty$.

\subsection{Janzer's Embedding Scheme for Graph Subdivisions}
In this subsection, we introduce (and slightly extend) two embedding schemes that were introduced by Janzer~\cite{janzer2018improved}.
Let us write $N_G(v_1,\cdots,v_k)$ to denote the common neighborhood of  $v_1,\cdots,v_k$ in $V(G)$,

\begin{definition}\label{red_Blue}
Given a graph $G$ with bipartition $X\cup Y$ and $a\in \mathbb{N}$, we define an edge-colored graph $J=J(G,a)$ as follows:
set $V(J)= X$ and $E(J)=\{uv:N_G(u,v)\geq 1\}$
and for each edge of $uv\in E(J)$, define
\begin{align*}
c(u,v)=
\begin{cases}
blue,& \text{if} \:\: 1\leq N_G(u,v)\leq a \\
red,& \text{if} \:\: N_G(u,v)\geq a.
\end{cases}
\end{align*}
\end{definition}

Given a colored graph $J = J(G,a)$, we write $R$ and $B$ for the graphs formed by the red and blue edges (respectively), and say that a subgraph $F \subset J$ is red if $F \subset R$, and blue if $F \subset B$.
Similarly, we write 
$N_B(v_1,\dots,v_k)$ for the common neighborhood of vertices  of  $v_1,\cdots,v_k \in V(J)$ in the graph $B$, and $d_B(v_1,\cdots,v_k)$ the size of this common blue neighborhood.
An important property of the colored graph $J(G,a)$ is recorded in the following lemma.

\begin{lemma}\label{no_red}
 Let $G$ be a bipartite graph, let $H$ be a graph, and set $a = e(H)$. If $J(G,a)$ contains a red copy of H, then $H^{\text{sub}} \subset G$.
\end{lemma}
\begin{proof}
Fix a red copy of $H$ in $J(G,e(H))$ and let
$\{u_1v_1,\dots,u_av_a\}$ be the edges of this copy of $H$.
We now greedily choose vertices $w_i\in N_G(u_i,v_i)$ for each $i\in [a]$, with $w_1,\dots, w_a$ all distinct.
This is possible because the edges $u_iv_i$ are all red, and gives us the desired copy of $H^{\text{sub}}$ in $G$.
\end{proof}

Using Lemma~\ref{no_red}, we can deduce that if G is $H^{\text{sub}}$-free and has sufficiently many edges, then $J(G,e(H))$ has many blue edges. To do so, we will use the following simple consequence of Tur\'an's theorem: if $G$ is a $K_r$-free graph with $n > r$ vertices, then $e(G) < {n \choose 2} - cn^2$ for some constant $c = c(r) > 0$.

\begin{lemma}\label{Janzer_blue} Fix a graph $H$ and let $G$ be an $H^{\text{sub}}$-free bipartite graph with bipartition $X\cup Y$.
Assume $e(G)/\big| Y\big| >  2v(H)$.
Then $J(G,e(H))$ has at least
$ce(G)^2/|Y|$ blue edges, where $c=c_H$ is a constant depending only on $H$.
\end{lemma}
\begin{proof}
Let $J=J(G,e(H))$ and let $y\in Y$.
Note that, since every pair in $N_G(y)$ has $y$ as a common neighbor $y$, the induced subgraph $J[N_G(y)]$ is a colored complete graph.
By Lemma~\ref{no_red}, there is no red $H$ in $J[N_G(y)]$ and consequently no red $K_{v(H)}$.
As noted above, by Tur\'an's theorem there exists a $c'=c'(H)>0$ such that if $d_G(y)> v(H)$,  then $J[N_G(y)]$ has at least $c'd_G(y)^2$ blue edges.

Note that at least half of the edges of $G$ are incident to vertices in Y of degree strictly greater than $v(H)$.
The number of pairs $(y,e)$ where $y\in Y$ and $d_G(y)>v(H)$ and $e$ is a blue edge of $J[N_G(y)]$ is therefore at least
\begin{align}\label{total_count}
              \sum_{v\in Y: d_G(y)> v(H)} c'd_G(y)^2
         \geq c'\bigg(\sum_{v\in Y:d_G(y)> v(H)} d_G(y) \bigg)^2/|Y|\geq  \frac{c'e(G)^2}{4|Y|},
 \end{align}
where the first inequality is obtained by using the  Cauchy-Schwarz inequality.
As each blue edge can only be counted $e(H)-1$ times in~\eqref{total_count}, it follows that the total number of blue edges is at least
$ \frac{c'e(G)^2}{4(e(H)-1)|Y|}.
$
Thus $J$ has at least $ce(G)^2/|Y|$ blue edges, where $c=c'/4(e(H)-1)$ is a constant depending only on $H$, as required.
\end{proof}

The following definition illustrates the embedding process based on the colored graph $J(G,a)$. Recall that if $uv$ is an edge in colored graph $J(G,a)$, then the vertices $u$ and $v$ have at least one common neighbor.

\begin{definition}\label{sbsgr}
Let $G$ be a graph with bipartition $X\cup Y$ and let $a\in \mathbb{N}$.
We call a subgraph $F$ of $J(G,a)$ \emph{helpful} if for all edges $uv$ in $F$, it is possible to choose a distinct common neighbour of $\{u,v\}$ in set $Y$, thus enabling an embedding of a copy of $F^{\text{sub}}$ into the graph $G$.
\end{definition}

In particular, the existence of an helpful copy $H$ in $J(G,a)$  implies that $H^\text{sub}$ embeds into $G$.
Note that, by Lemma~\ref{no_red}, every red copy of $H$ in $J(G,e(H))$ is helpful.

We summarize one more idea from~\cite{janzer2018improved} as a proposition below.
Given a graph $F$ and an integer $t$, let $K_{F,t}$ denote the graph formed by a copy of $F$ and an independent set $T$ with $t$ vertices, where $V(K_{F,t})=V(F)\cup V(T)$ and $$E(K_{F,t})=E(F)\cup \{(u,v):u\in V(F),v\in V(T)\}.$$
In particular, note that $L_{s,t}$ is equal to $K_{F,t}$ with $F=K_{s-1}$,
Janzer (see Theorem 4 in \cite{janzer2018improved}) showed that if the colored graph has a helpful blue $K_{s-1}$ with certain properties, then there exists a helpful blue $L_{s,t}$.
We need to generalize the result to an arbitrary graph $F$; fortunately, this follows from essentially the same proof.
In our case, we require $F$ to be an arbitrary graph.

Observe that if a graph $H$ has an independent set $T$ of size $t$, then $H$ is a subgraph of $K_{F,t}$ for some graph $F$.
Our strategy will be to first find a helpful blue copy of $F$ with certain extra properties, and then extend this to a helpful blue copy of $K_{F,t}$ using the following proposition.


\begin{prop}[See Theorem 4 of~\cite{janzer2018improved}]\label{End_of_Janzer}
Given a graph $F$, an integer $t\in \mathbb{N}$ and $K\geq 1$, there exists a constant $c=c(F,t,K)>0$ such that the following holds. Let $G$ be a $K$-almost regular bipartite graph with bipartition $X\cup Y$.
 If $J=J(G,e(K_{F,t}))$ contains a helpful blue copy of $F$ with
 $V=V(F)$, such that the following conditions are satisfied:
\begin{enumerate}
   \vspace{0.5mm}
   \item  $J[V]$ is blue;
 \item  $d_B(V)\geq c\delta(G)$,
\end{enumerate}
then we can find a set  $Z\subset N_B(V)$ of size $t$, such that $J[V\cup Z]$ contains a helpful blue copy of $K_{F,t}$.
\end{prop}

\begin{proof}
Let $V=\{x_1,\cdots,x_s\}$ and let $\delta(G)=d$, $c=K\big({s\choose 2} + (s-1)t\big)(e(K_{F,t})-1)+1$.
The assumption (1) ensures
each pair of $V$ is either blue or non-colored.
In particular, for any pair $\{x_i,x_j\}$,  $|N_G(x_i,x_j)|\leq e(K_{F,t})-1$ by definition.
Considering all pairs in $V$, the union
\begin{equation}
Y^{(0)}:=\bigcup_{1\leq i< j \leq s} N_G(x_i,x_j)
\end{equation}
has a size at most
${s\choose 2}\big(e(K_{F,t})-1\big)$.
Since $G$ is $K$-regular, the total number of vertices in $X$ that are adjacent to some vertex in $ Y^{(0)}$ is at most $Kd |Y^{(0)}|$.
By removing these vertices from $N_B(V)$, we obtain a set
$$Z^{(0)}:=N_B(V)\backslash \bigcup_{w\in Y^{(0)}}N_G(w),$$
and $Z^{(0)}$ is not empty since $d_B(V)\geq cd$, by our choice of $c$.

Let $z_1\in Z_0$, and note that  $N_G(z_1)\cap Y^{(0)}$ is empty.
Moreover, since $z_1\in N_B(V)$, there is a blue copy of $K_{F,1}$ with vertex set $V\cup z_1$.
We claim this blue copy of $K_{F,1}$ is helpful. To see this, suppose first that the pairs $\{z_1,x_i\}$ and $\{z_1,x_j\}$
share a common neighbor $y$ for some $1\leq i, j\leq s$.
Then $y\in N_G(z_1,x_i,x_j)\subset N_G(z_1)\cap Y^{(0)}$, contradicting our choice of $z_1$.
By the same argument, the pairs $\{z_1,x_i\}$ and $\{x_i,x_j\}$ cannot share a common neighbor. Since we assumed that $F$ is helpful, it follows that $K_{F,1}$ is helpful, as claimed.

In other words, we have successfully identified a helpful blue copy of $K_{F,1}$ in $J[V\cup\{z_1\}]$.
Now, let's assume that we have already located a helpful blue $K_{F,m}$ in $X$ along with additional vertices denoted as $z_1,z_2,\cdots,z_m$ for $m<t$.
We proceed by defining $Y^{(m)}$ as follows:
$$Y^{(m)}:=Y^{(m-1)}\cup\bigg( \bigcup_{1\leq i\leq s}N_G(z_m,x_i)\bigg)=Y^{(0)}\cup\bigg( \bigcup_{1\leq k\leq m,1\leq i\leq s}N_G(z_k,x_i)\bigg),$$
The size of $Y^{(m)}$ is at most $\big({s\choose 2} + sm\big)(e(K_{F,t})-1)$.
Since $G$ is $K$-regular, the set of vertices that are adjacent to some vertex in $Y^{(m)}$ is bounded by $Kd|Y^{(m)}|$. By removing these vertices from $N_B(V)$, we obtain
$$Z^{(m)}:=N_B(V)\backslash \bigcup_{w\in Y^{(m)}}N_G(w).$$
Note that the  size of $Z^{(m)}$ is still not empty by our choice of $c$.

Now, we can select the vertex $z_{m+1}$ from $Z^{(m)}$.
Then $Y^{(m)}\cap N_G(z_{m+1})$ is empty.
As $z_{m+1}x_i$ is blue for each $i$, we obtain a blue $K_{F,{m+1}}$. Furthermore,
 we claim the pairs $\{z_{m+1}, x_i\}$ and $\{z_k, x_j\}$ do not share any common neighbour, for all $k\in [m]$ and $i,j\in [s]$.
Otherwise, we would find a vertex
$$y\in N_G(z_{m+1},z_k,x_i,x_j)\subset Y^{(k)} \cap N_G(z_{m+1}), $$
which is a contradiction since $Y^{(k)} \subset Y^{(m)}$.
By adding $z_{m+1}$ to embedded vertices and blue edges $\{z_{m+1},x_i\}$ for all $i\in[s]$ to embedded edges, we successfully find a helpful blue $K_{F,m+1}$.
Note that $$d_B(V)\geq cd >
K\bigg({s\choose 2} + (s-1)t\bigg)(e(K_{F,t})-1)d\geq Kd|Y^{(t-1)}|,$$ we can continue this process iteratively until we achieve a helpful blue copy of $K_{F,t}$.\end{proof}

In order to find a set $V$ satisfying condition (1) of Proposition~\ref{End_of_Janzer}, we will use the following simple supersaturation lemma from~\cite{conlon2021more}.
In the proof of the following lemma, we need use the fact that the Ramsey number $R(k,s)$,
which is defined to be the minimum integer $n$  such that any red-green edge-coloring of the complete graph on $n$ vertices  contains either a red clique of size $k$ or a green clique of size $s$, is finite.

\begin{lemma}\label{Ramsey_stuff}
Fix $k,s\in \mathbb{N}$,
and let $J$ be a red-blue edge-colored graph. If $J$ does not have any red clique of size $k$, then every sufficiently large subset $X\subset V(J)$ contains $\Omega(|X|^s)$ sets $X_1\subset X$ of size $s$ such that $J[X_1]$ is blue.
\end{lemma}

\begin{proof}
Recoloring blue edges and non-adjacent pairs green, we obtain a red-green edge coloring of $K_{v(J)}$.
Then by the definition of Ramsey number,
 for any subset $X_1\subset X$ with $|X_1|=R(k,s)$,  in $J[X_1]$ there is either a red $k$-clique, or a green $s$-clique. Since $J[X_1]$ does not have any red $k$-clique, we can conclude that $J[X_1]$ contains a green $s$-clique.
 Note that each green $s$-clique is contained in
${|X|-s \choose R(k,s)-s}$
distinct $R(k,s)$-subsets of $X$ and each green $s$-clique induces a blue $J[V]$.
So we find a family $\mathcal X$ consisting of such green $s$-cliques, and
\begin{equation}
|\mathcal{X}| \geq {|X|\choose R(k,s)} \Big/ {|X|-s \choose R(k,s)-s} = \Omega(|X|^s).
\end{equation}
The proof is completed.
\end{proof}

\section{Proof of Theorem \ref{thm_2}}
In this section we prove Theorem~\ref{thm_2}. We borrow some arguments from section 4 in~\cite{conlon2021more}, combined with additional new ideas.

The reduction Lemma~\ref{make_assumption} allows
us to assume our host graph $G$ satisfies several conditions below.
$G$ has $n$ vertices, has a balanced vertex bipartition $X\cup Y$.
Moreover, $G$ is $\mathcal (K_{s+1,t}^+)^{\text{sub}}$-free, and $K$-almost regular with a minimum degree satisfying
\begin{equation}\label{degree_condition_again}
d=\omega ( n^{(s^2+1)/(2s+2s^2)}).
\end{equation}
On set $X$,
we define the colored graph $J=J(G,e(K_{s+1,t}^+))$ following Definition~\ref{red_Blue}.

We will utilize Lemma~\ref{Janzer_blue} twice.
Firstly,
within the bipartition $X\cup Y$, the average degree condition is satisfied as $d|X|/|Y|=\Omega(d)=\omega(1)$.
This enables us to identify a vertex $x\in X$, along with its blue neighborhood $X_1=N_B(x)$, containing $\Theta(d^2)$ vertices.
The lower bound is justified by Lemma~\ref{Janzer_blue}, while the upper bound is supported by the property of $G$ being $K$-almost regular.
Secondly, in the bipartition $X_1 \cup Y$,
 the average degree condition still holds true, with $d |X_1|/|Y|=\Omega( d^3/n)=\omega(1)$.
Applying Lemma~\ref{Janzer_blue} again, we conclude that
the number of blue edges in $J[X_1]$ is $\Omega(d^6/n)$.

We stress that the vertex $x$ in the previous paragraph is fixed from now on.

Let us write $K_{u_1,\cdots,u_s}^v$ to represent a copy of $s$-star
centered at $v$ with leaves $u_1, u_2,\cdots$ and $u_s$.
The following lemma gives the correct count
for the total number of the  helpful blue $s$-stars $K^y_{u_1,\cdots,u_s}$ found in $J[X_1]$,
with the additional property that
 the star $K^{x}_{u_1,\cdots,u_s}$ is also blue and helpful.

\begin{lemma}\label{super_helpful}
Define $\mathcal P$ to be the family  of $(s+1)$-tuples, consisting of
$(y,u_1,\cdots,u_s)\in X_1^{s+1}$ with the following properties.
\begin{enumerate}[(i)]
\item $K^y_{u_1,\cdots.u_s}$ is a blue and helpful $s$-star.
\item $K^{x}_{u_1,\cdots,u_s}$ is a blue and helpful $s$-star.
\end{enumerate}
Then $|\mathcal P|=\Omega(d^{4s+2}/n^s)$.
\end{lemma}

\begin{proof} 
Let us write $d_B^{X_1}(y)=|N_B(y) \cap X_1|$, the size of the blue neighborhood of vertex $y$ in $J[X_1]$.
The number of blue edges in $J[X_1]$ can then be represented as $\bigg(\sum_{y\in X_1} d_B^{X_1}(y)\bigg)/2$, and this quantity is $\Omega(d^6/n)$.

For some small constant $\varepsilon$, we write
\begin{equation}\label{sigma_good_y}
I =\{y\in X_1 \big| d_B^{X_1}(y) \geq \varepsilon d^4/n \}.
\end{equation}
Given that $|X_1|=\Theta ( d^2)$, we can conclude that $\sum_{y\in X_1\setminus I} d_B^{X_1}(y) \leq \varepsilon d^6/n$.
Subsequently, the total number of blue edges inside $X_1$ connected to vertices in $I$ remains of the same order, that is,
\begin{equation}\label{good_y_sum}
\sum_{y\in I} d_B^{X_1}(y) = \Omega(d^6/n).
\end{equation}

Let us recall that $N_B(x,y)$ denotes the common blue neighborhood of $x$ and $y$ in $X$.
For any $y\in I$,
we see that $|N_B(x,y)|=d_B^{X_1}(y)\geq\varepsilon d^4/n $.
Now, we are going to search for $s$-tuples $(u_1,\cdots,u_s)$ that, together with
$y$, satisfy both $(i)$ and $(ii)$.

We list both $N_G(x)=\{b_{x,1},\cdots,b_{x,k_x}\}$ and $N_G(y)=\{b_{y,1},\cdots,b_{y,k_y}\}$ in such a way that the common neighbors of $x$ and $y$ appear at the beginning in both lists.
In other words, there exists some $k'\leq \min \{k_x,k_y\}$ such that, for $i=1,\cdots,k'$, $b_{x,i}=b_{y,i}$.
For all $1\leq i\leq k_x, 1\leq j\leq k_y$, we define the sets $W_{i,j}=\big( N_G(b_{x,i}) \cup N_G(b_{y,j}) \big) \cap N_B(x,y)$.
It is obvious that $\bigcup_{i=1}^{k_x} \bigcup_{j=1}^{k_y} W_{i,j}\subseteq N_B(x,y).$
Conversely, each vertex $v$ in $N_B(x,y)$ shares at least one common neighbor with $x$ and $y$, respectively.
Therefore, for some $i$ and $j$,
vertex $v$ belongs to $N_G(b_{x,i})$ and $N_G(b_{y,j})$, both of which are a subset of $W_{i,j}$. Consequently, we conclude that \begin{equation}
N_B(x,y)=\bigcup_{i=1}^{k_x} \bigcup_{j=1}^{k_y} W_{i,j}.
\end{equation}
Now we re-order these sets such that $W_{1,1},W_{2,2}\cdots,W_{k',k'}$ appear first, and proceed to eliminate duplicate sets and empty sets one by one.  For the first $k'$ sets $W_{i,i}$, $1\leq i \leq k'$, let us assume that there are $k^\ast$ sets left. This allows us to write a family $\mathcal W$ consisting of these re-ordered sets as follows:
 \begin{equation}
 \mathcal W=\{U_1,\cdots, U_{k}\},\ k\geq k^{\ast},
 \end{equation}
 Note that for each $i\leq k^{\ast}$, $U_i$ corresponds to $W_{j,j}$ for some $j$, and it follows that $b_{x,j}=b_{y,j}$, consequently $U_i=N_G(b_{x,j})  \cap N_B(x,y)= N_G(b_{y,j}) \cap N_B(x,y).$
 Moreover, deleting the duplicate sets and empty sets will not change the union, so we still have
  \begin{equation}
  N_B(x,y)=\bigcup_{i=1}^{k} U_i.
 \end{equation}

{\bf Claim:}
There exists a partition $\bigcup_{j=1}^\ell P_j$ of $N_B(x,y)$ with the following properties:
\begin{enumerate}
\item $\ell= \Theta(d_B^{X_1}(y)/d)$;
\item for each $1\leq j\leq \ell$, $|P_j|=\Theta(d)$;
\item for any $u_1,\cdots,u_s$ taken from sets $P_{i_1}, \cdots, P_{i_s}$,
with $1\leq i_1<\cdots<i_s\leq \ell$, both $K^x_{u_1,\cdots,u_s}$
and $K^y_{u_1,\cdots,u_s}$ are helpful.
\end{enumerate}
{\it \noindent Proof of the Claim:}
We show the claim via an induction process.
First, we observe that each element $U_i$ in the family $\mathcal W$ is a set with size at most $2Kd$. This is because $U_i$ is essentially a subset of a union of the neighborhoods of two vertices in $Y$.
Now, we update $U_1=U_1$ and $U_i= U_i\setminus \bigg(\bigcup_{j=1}^{i-1} U_j\bigg)$ for $i=2,\cdots,k$.
This ensures that the sets $U_i$ are disjoint from each other.
With this setup, $\bigcup_{i=1}^k U_i$
forms a partition of $N_B(x, y)$, and the size of each part in this partition is at most $2Kd$.

We can choose some $k_1$
such that $P_1=\bigcup_{j=1}^{k_1}U_j$ has size $[\frac 12 Kd, \frac 52 Kd]$.
Now any vertex $v\in N_B(x,y) \backslash P_1$ connects to some $b\in B$ which is not the common neighbor of any vertex in $P_1$ and $x$ (or $y$).
Inductively, suppose we have found disjoint sets
 $P_1,\cdots,P_{t}$
 such that each of them has size in the interval $[\frac 12 Kd, \frac 52 Kd]$,
 $\bigcup_{i=1}^t P_i= \bigcup_{i=1}^{k_t} U_{i}$ for some $k_t$,
 and any $u_1\in P_{i_1}, \cdots, u_s\in P_{i_s}$ for any sequence
 $1 \leq i_1 <\cdots <i_s \leq t$
make both $K^x_{u_1\cdots,u_s}$ and $K^y_{u_1,\cdots, u_s}$ helpful.
Now we are dealing with two cases, namely, $k_{t} \leq k^\ast$ or $k_{t}>k^\ast$.
However, there is no difference in both cases, that is,
any vertex $z_{t}$ in $\bigcup_{j=k_t+1}^k U_j= N_B(x,y) \backslash \bigcup_{j=1}^{t}P_j$
does not share any common neighbor with
any vertex in $\bigcup_{j=1}^{t}P_j$.

Now we can find some $k_{t+1}$ such that the set
$P_{t+1}=\bigcup_{i=k_t+1}^{k_{t+1}}U_i$ has size in
$[\frac 12 K d, \frac 52 Kd]$.
If not, which means that the number of the rest vertices is at most
$\frac 12 Kd$, we include all the rest vertices into $P_{t+1}$, define $\ell=t+1$, and the process terminates.
If the rest of the vertices is more than $\frac{1}{2}K d$, we continue with the process. Eventually we will terminate since there are only finitely many vertices.
Now $\bigcup_{j=1}^\ell P_j$ is a partition of $N_B(x,y)$ satisfying (2) and (3), and $\ell=d_B^{X_1}(y)/\Theta(d)=\Theta(d_B^{X_1}(y)/d)$.
\qed

After establishing the claim, we can proceed with the counting.
We have a list of $\ell$ pairwise disjoint sets in $N_B(x,y)$, each of which contains $\Theta(d)$ vertices.
We choose $s$ sets
$P_{i_1}, \cdots, P_{i_s}$ from this list, which can be done in $\ell \choose s$ ways.
And within each of these selected $s$ sets, we can choose arbitrary element $u_i\in P_i$ arbitrarily.
Therefore, there are at least $\Omega \Big( {\ell \choose s}   \times d^s \Big)=\Omega \big((d_B^{X_1}(y))^s \big)$ many $s$-tuples $(u_1,\cdots,u_s)$
making both
$K^x_{u_1,\cdots,u_s}$ and
$K^y_{u_1,\cdots,u_s}$ helpful.

To sum up, we have shown that for each $y\in I$, the total number of $s$-tuples satisfying (i) and (ii) is $\Omega\big((d_B^{X_1}(y))^s\big)$.
Recalling (\ref{sigma_good_y}) and (\ref{good_y_sum}), and applying Jensen's inequality,
it follows that the number of $(s+1)$-tuples  $(y,u_1,\cdots,u_s)$ satisfying both (i) and (ii) is
\begin{equation}
\Omega\bigg(\sum_{y\in I} d_B^{X_1}(y)^s\bigg)=
\Omega\bigg( |I| \bigg(\frac{1}{|I|}\sum_{y\in I} d_B^{X_1}(y)\bigg)^s \bigg)=\Omega\bigg(d^{4s+2}/n^s\bigg).
\end{equation}
The proof is completed.
\end{proof}

Below, for any $s$-tuple $(u_1,\cdots,u_s) \in X_1^s$,
 let $P'_{u_1,\cdots,u_s}(X_1)$  be the set of vertices $y\in X_1$
such that both properties (i) and (ii) in Lemma~\ref{super_helpful} hold true.

\begin{lemma}\label{many_s_cliques_2}
Suppose there exists an $s$-tuple
$(u_1,\cdots,u_s)\in X_1^s$ so that $|P'_{u_1,\cdots,u_s}(X_1)|= \omega(1)$.
Then we can find a vertex $b\in B$ and
a family $\mathcal V$ of $s$-tuples in $X_1^s$ with $|\mathcal V|=\Omega(|P'_{u_1,\cdots,u_s}(X_1)|^s)$, such that every $\{v_1,\cdots,v_s\}$ in $\mathcal V$ forms a non-red $s$-clique in $J[X_1]$,
 and one of the following two cases happens:
\begin{enumerate}
\item Either $\{v_1,\cdots,v_s, x\} \subset N_G(b)$.
\item Or there is a $k\in \{1,\cdots.s\}$, such that $\{v_1,\cdots,v_s, u_k\} \subset N_G(b)$.
\end{enumerate}
\end{lemma}

\begin{proof} Let's define a maximal subset
$\{y_1,\cdots,y_r\} \subset P'_{u_1,\cdots,u_s}(X_1)$ such that,
together with $\{x\}$ and  $\{u_1,\cdots,u_s\}$, they form a helpful copy of $K_{s+1,r}^+$, denoted by $H$.
More precisely, the edge set $E(H)$
consists of $\{x,u_1\}$, $\{ \{x,y_i\} \}_{1\leq i\leq r}$, and
$\{ \{u_i,y_j\}\}_{1\leq i\leq s, 1\leq j\leq r}$. Because $H$ is helpful, for each $e\in E(H)$, we can find one distinct vertex $b\in B$ such that   it is a common neighbor of both vertices of the edge $e$. Let $B_{used}$ be the set of these vertices. Note that  $|B_{used}|=1+(s+1)r$.
Since $G$ is $(K_{s+1,t}^+)^{\text{sub}}$-free, we must have $0\leq r<t$, where $r=0$ means that
the maximal subset could be an empty set.

For any $y\in P'_{u_1,\cdots,u_s}(X_1)\backslash \{y_1,\cdots,y_r\}$,  by maximality of the set $\{y_1,\cdots,y_r\}$,
there exists some $b\in B_{used}$ such that $y\in N_G(b)$.
By pigeonhole principle, there exists a certain vertex
$b^\ast\in B_{used}$ which is connected by at least $\Omega(|P'_{u_1,\cdots,u_s}(X_1)|)$ vertices in $P'_{u_1,\cdots,u_s}(X_1)$.
Note that either $b^\ast \in N_G(x)$,
or there is some $1\leq k\leq s$
such that $b^\ast\in N_G(u_k)$, which will lead to the dichotomy in the statement of this lemma.
Applying Proposition~\ref{Ramsey_stuff},
there exist at least $\Omega(|P'_{u_1,\cdots,u_s}(X_1)|^s)$ many non-red $s$-tuples $\{v_1,\cdots,v_s\}$ satisfying both (1) and  (2). Let $\mathcal{V}$ denote the family of $s$-tuples.
So the proof is complete.
\end{proof}

The proof of Theorem~\ref{thm_2} reduces to the following proposition.
\begin{prop}\label{s_tuple_again}
Let $X_1$ be the blue neighborhood of $x\in X$ such that $|X_1|=\Theta(d^2)$.
Then there exists a set $\{y_1,\cdots,y_s\} \subset X_1$ such that $J[y_1,\cdots,y_s]$ is blue, and
the common blue neighborhood of $\{y_1,\cdots,y_s\}$ in $J[X_1]$ has size of $\omega(d)$.
\end{prop}

\begin{proof}
We start from Lemma~\ref{super_helpful} and obtain that
\begin{equation}
|\mathcal P|=\sum_{(u_1,\cdots,u_s)\in X_1^s}|P'_{u_1,\cdots,u_s}(X_1)|=\Omega(d^{4s+2}/n^s)
\end{equation}
We also notice that
the total number does not drop too much if we only sum over all the $s$-tuples $u_1,\cdots,u_s$
with $|P'_{u_1,\cdots,u_s}(X_1)|=\omega(1)$.
More precisely,  let
\begin{equation}I=\{(u_1,\cdots,u_s)\in X_1^s | (u_1,\cdots,u_s) \text{ satisfies }|P'_{u_1,\cdots,u_s}(X_1)|=\omega(1)\}.
\end{equation}
Recall that $|X_1|=d^2$, where $d$ satisfying~(\ref{degree_condition_again}), we know that
$$\sum_{(u_1,\cdots,u_s)\in X_1^s \setminus I}|P'_{u_1,\cdots,u_s}(X_1)|=O(|X_1|^{s}),$$ which is $o(d^{4s+2}/n^s).$ Consequently, we have
\begin{equation}\label{P'}
\sum_{(u_1,\cdots,u_s)\in I}|P'_{u_1,\cdots,u_s}(X_1)|=\Omega( d^{4s+2}/n^s).
\end{equation}

Then by Lemma~\ref{many_s_cliques_2}, for each such $s$-tuple, we can obtain $b\in Y$, and a family $\mathcal V$ of $s$-tuples forming
non-red $s$-cliques, with $|\mathcal V|=\Omega(|P'_{u_1,\cdots,u_s}(X_1)|^s)$, such that
\begin{enumerate}
\item either
$\{v_1,\cdots,v_s\} \in \mathcal V$ implies $\{v_1,\cdots,v_s,x\} \subset N_G(b)$,
\item or there exists some $u_k$ such that
$\{v_1,\cdots,v_s\} \in \mathcal V$ implies $\{v_1,\cdots, v_s, u_k\}\subset N_G(b)$.
\end{enumerate}
By combining the estimation in~\ref{P'} and Jensen's inequality, we can determine the total number  of $2s$-tuples $(u_1,\cdots,u_s,v_1,\cdots,v_s)$
with the above properties

\begin{align}
\Omega\bigg(\sum_{(u_1,\cdots,u_s)\in I}|P'_{u_1,\cdots,u_s}(X_1)|^s \bigg)
&= \Omega\bigg( |I| \bigg(\sum_{(u_1,\cdots,u_s)\in I}\frac{|P'_{u_1,\cdots,u_s}(X_1)|}{|I|}\bigg)^s \bigg) \\
&=\Omega\bigg( |I|^{1-s} \bigg(\sum_{(u_1,\cdots,u_s)\in I}|P'_{u_1,\cdots,u_s}(X_1)|\bigg)^s\bigg)
\nonumber \\
&=\Omega\bigg( d^{2s(-s+1)}\times (d^{4s+2}/n^s)^s\bigg)
\nonumber \\
&=\Omega\bigg(d^{2s^2+4s}/n^{s^2}\bigg). \nonumber
\end{align}
 Then, at least one of the above dichotomies will be satisfied by at least $\Omega(d^{2s^2+4s}/n^{s^2})$ many $2s$-tuples.
In both cases, we will do double counting, with the degree condition
(\ref{degree_condition_again}) at hand.

If assertion (1) is satisfied by $\Omega(d^{2s^2+4s}/n^{s^2})$ many $2s$-tuples, we can choose
$b\in B$ in $n$ ways, then we choose $v_1,\cdots,v_s \in N_G(b)$ in $d^s$ ways.
We note the degree condition implies that
$d^{2s^2+4s}/n^{s^2} = \omega(d^{s} nd^{s}).$
 Therefore, by pigeonhole principle, given by a certain
one particular such choice of $b, v_1,\cdots, v_s$, there will be at least $\omega(d^s)$ ways to extend this choice to find
$2s$-tuple to satisfy assertion (i). Since there are $s$ coordinates to consider, we conclude that
the total number of vertices showing up in these $s$-coordinates is at least $\omega(d)$.

If assertion (2) is satisfied by $\Omega(d^{2s^2+4s}/n^{s^2})$ many $2s$-tuples, we can
choose $b \in B$ in $n$ ways,  and then choose $1\leq k\leq s$, and then choose $v_1,\cdots, v_s,u_k$ in $d^{s+1}$ ways.
Use the degree condition again, we see $d^{2s^2+4s}/n^{s^2} = \omega(d^{s-1}n d^{s+1})$. So by pigeonhole principle,
there is a certain choice of $b,u_k,v_1,\cdots, v_s$, which extends in at least $\omega(d^{s-1})$ ways to a different $2s$-tuple satisfying assertion (ii). Since there are $s-1$ coordinates to consider,
we conclude that the total number of vertices showing up in these
$(s-1)$-coordinates is at least $\omega(d)$. Now the proof of the Proposition is completed.
\end{proof}

\begin{proof}[End of Proof of Theorem~\ref{thm_2}]
From the initial part of this section, we obtain  a vertex $x$ with the blue neighborhood $X_1=\Theta(d^2)$ in the required host graph $G$.
Proposition~\ref{s_tuple_again} provides with an $s$-subset of vertices
$\{v_1,\cdots,v_s\}\subset X_1$, which, together with the vertex $x$,
forms a non-red $(s+1)$-clique. Let $H$ be the graph that consists of vertex set $\{x,v_1,\cdots,v_s\}$
and only one edge $\{x,v_1\}$. So in $J=J(G,e((K_{s+1,t})^+))$, there is a helpful and blue copy of $H$. Then we apply Proposition~\ref{End_of_Janzer} to find a helpful copy of
$K_{s+1,t}^+$ in $J$, and the proof is completed.
\end{proof}

\section{proof of theorem~\ref{family_version}}

Fix an integer $s\geq 3$. By applying Lemma~\ref{make_assumption},
we may assume that the host graph  $G$ has $n$ vertices and is bipartite,  with a balanced bipartition $X\cup Y$.
Moreover, $G$ is $\mathcal F_{1,s,t}$-free and $K$-almost regular, with  minimum degree
\begin{equation}\label{assumption_on_delta_family}
d=\omega( n^{(s^2+1)/(2s+2s^2)}).
\end{equation}

In this context, we will frequently need to
count the total number of copies of helpful blue $K_{s,1}$ subgraphs
in certain colored subgraph.
This is addressed in the following lemma. By Definition~\ref{red_Blue}, let $J=J(G, e(\mathcal F_{1,s,t}))$, where $e(\mathcal F_{1,s,t})=\max\{e(H)|H\in \mathcal F_{1,s,t}\}$.
\begin{lemma}\label{helpful_imhelpful}
Suppose that in $J$,
there exists a vertex $x\in X$ and a subset $S\subset N_B(x)$ of its blue neighborhood such that 
$|S| = \omega(d)$.
Then the number of helpful blue $s$-stars centered at
$x$, with all leaves contained in $S$, is $\Omega( |S|^s)$.
\end{lemma}

\begin{proof}
Let us list the neighbors of $x$ in $G$ as 
$N_G(x)=\{b_1,\cdots,b_k\}$. For each $1\leq i\leq k$, define 
$$W_i:= N_G(b_i) \cap Y.$$ 
For simplicity, we discard any empty sets from this collection, rename the remaining non-empty sets, and continue to denote them $W_i$.
Since every blue neighbor of $x$ shares at least one common neighbor with $y$ in $Y$,
it must belong to some $N_G(b_i)$. Therefore, we have $S=\bigcup_{i=1}^k W_i$.

We claim that
 there exists a disjoint union $S= \bigcup_{i=1}^\ell P_i$,
 where each set satisfies $|P_i|=\Theta(d)$, and
$\ell=\Theta(|S|/d)$, such that the following holds: if the leaves of an  $s$-star are chosen from $s$ distinct sets $P_{i_1}, P_{i_2}, \cdots, P_{i_s}$, respectively,
then the corresponding copy of the $s$-star is helpful.

To prove the claim, we first redefine  the sets $W_i$ to make them disjoint. Set
$W_1=W_1$, and for each $i=2,\cdots,k$, define 
\begin{equation}
W_i= W_i\setminus \bigg(\bigcup_{j=1}^{i-1} W_j\bigg).
\end{equation}
Then
the disjoint union 
$S= \bigcup_{i=1}^k W_i$
forms a partition of $S$, where each part has size at most $Kd$, since $G$ is $K$-almost regular.
Note that every vertex $w_i \in W_i$ has the property that it shares a common neighbor with $y$ that is not a common neighbor of $y$ and any vertex in $W_{i'}$ for $i'<i$.	
Now, we define sets $P_i$ by grouping the $W_i$'s into blocks of approximately $Kd$ elements. First,  find an index $i_1$,
such that 
$P_1= \bigcup_{j=1}^{i_1}W_j$
satisfies 
\begin{equation}
|P_1|\in [\frac 12 K d, \frac 32 K d].
\end{equation}
Suppose inductively that we have defined disjoint unions $P_1,\ldots,P_t$, where
$P_j=\bigcup_{i=i_{j-1}+1}^{i_j}W_i$
for some sequence $1 \leq i_1<\cdots <i_t<k$, and each $P_j$
satisfies 
\begin{equation}
|P_j|\in [\frac 12 K d,\frac 32 Kd].
\end{equation}
If $
S\backslash \bigg( \bigcup_{j=1}^t P_j \bigg)|\leq \frac 12 K d$,
we simply include it into $P_t$ and set $\ell=t$, terminating the process.
Otherwise, we can find $i_{t+1} < k$
such that 
$P_{t+1}=\bigcup_{j=i_t+1}^{i_{t+1}} W_j$
has size in $[\frac 12 K d, \frac 32 Kd]$.
This process must terminate since $S$ is finite. At the end, each $P_i$ satisfies
\begin{equation}
|P_i|\in [\frac 12 K d, 2Kd]
\end{equation} 
and the claim follows.
Clearly, $\ell=\Theta(|S|/d)$.

We now observe that there are $ {\ell \choose s} $ ways to choose $s$ distinct sets $P_i$ from the partition constructed above. In each selected $P_i$,
there are $\Theta(d)$ choices for selecting one vertex. Choosing one vertex from each of the $s$ distinct sets yields a helpful $s$-star.
It follows that the number of such helpful blue $s$-stars is at least
\begin{equation}
\Omega\big( { |S| /d \choose s} \times d^s \big),
\end{equation} 
which simplifies to $\Omega( |S|^s)$. This conclude the proof.
\end{proof}

For any subset \( X_1 \subseteq A \), let \( J[X_1] \) denote the subgraph of \( J \) induced by the vertex set \( X_1 \).
As in the proof of Theorem~\ref{thm_2}, we apply Lemma~\ref{Janzer_blue} twice.
There exists a vertex $x$ whose blue neighborhood $X_1$ has size $|X_1| =\Theta(d^2)$. Since $d^3/n=\omega(1)$, the number of blue edges in $J[X_1]$ is $\Omega(d^6/n)$.
For any $u_1,\ldots, u_s \in X_1$, let $K_{u_1,\ldots,u_s}^y$ denote an $s$-star with center $y$ and leaves $u_1,\cdots, u_s$.
Define the set
\begin{equation}\label{degree_of_stuple}
P_{u_1,\cdots,u_s}(X_1): = \{ y\in X_1\big| K_{u_1,\cdots,u_s}^y \text{ is a helpful blue star}\}.
\end{equation}
In other words, $|P_{u_1,\cdots,u_s}(X_1)|$ counts the number of helpful blue $s$-stars associated with a fixed $s$-tuple in $X_1$. 
In what follows, we will need to recall Definition~\ref{sbsgr}, which defines when a colored copy of a graph is considered helpful.

\begin{lemma}\label{helpful_star_count}
\begin{equation}
\sum_{u_1,\cdots,u_s \in X_1} |P_{u_1,\cdots,u_s}(X_1)| =\Omega(d^{4s+2}/n^s).
\end{equation}
\end{lemma}

\begin{proof}
We use $d_B^{X_1}(y)=|N_B(y)\cap X_1|$
to denote the number of blue neighbours of $y$ in $X_1$. Then, \begin{equation}
\sum_{y\in X_1}d_B^{X_1}(y)= \Omega(d^6/n).
\end{equation}
Fix a small constant $\varepsilon$, and consider only those vertices $y_1,\cdots,y_\ell \in X_1$
that have at least $\varepsilon d^4/n$ blue neighbors in $X_1$.
Since $|X_1|=\Theta(d^2)$, the remaining vertices contribute at most
$\varepsilon' d^6/n$ blue edges in $J[X_1]$, where $\varepsilon'$ depends on $\varepsilon$.
Therefore, if $\varepsilon$ is sufficiently small, the total number of blue edges in $J[X_1]$ remains asymptotically unchanged:
\begin{equation}
\sum_{i=1}^\ell d_B^{X_1}(y_i) =\Omega(d^6/n).
\end{equation}
Since each $y_i$ satisfies $d_B^{X_1}(y_i) \geq \varepsilon_0 d^4/n=\omega(d)$, Lemma~\ref{helpful_imhelpful} implies that
 each $y_i$ is the center of at least
 $\Omega\big((d_B^{X_1}(y_i))^s\big)$  helpful blue $s$-stars in $X_1$.
  By Jensen's inequality, the total number of helpful blue $s$-stars is at least
 \begin{equation}
 \Omega \Big( \sum_{i=1}^\ell d_B^{X_1}(y_i)^s \Big ) =
 \Omega\Big( \ell  \times (\frac 1\ell \sum_{i=1}^\ell d_B^{X_1}(y_i))^s \Big)=
 \Omega (d^{4s+2}/n^s),
 \end{equation}
where we have also used the fact  that $\ell\leq |X_1|=O(d^2)$. This completes the proof.
 \end{proof}

Observing in the context of
Proposition~\ref{End_of_Janzer}, if we find a non-red $s$-clique in $X_1$ whose common blue neighborhood
intersects $X_1$ in a set of size $\omega(d)$, then we can identify a helpful copy of $K_{s,t}$ in $J[X_1]$ for any sufficiently large constant $t$.
Specifically, by selecting common neighbors of the vertex $x$ and each vertex in $V(K_{s,t})$, this would imply that $G$ contains a graph from $\mathcal F_{1,s,t}$, contradicting our initial assumption about $G$.
Indeed, the proof of Theorem~\ref{family_version} reduces to the following proposition.

\begin{prop}\label{good_pair_in_Nbx}
There exist distinct vertices $v_1,\cdots,v_s \in X_1$ satisfying the following properties.
\begin{enumerate}
\item  The subgraph $J[v_1,\cdots,v_s]$ is either blue or empty.
\item The number of vertices
$z \in X_1$ for which $K_{v_1,\cdots,v_s}^z$ forms a blue star is of order $\omega(d)$.
\end{enumerate}
\end{prop}

\begin{proof}
Recall that $|X_1|=\Theta(d^2)$, and
by Lemma~\ref{helpful_star_count}, \begin{equation}
	\sum_{u_1,\cdots,u_s \in X_1} |P_{u_1,\cdots,u_s}(X_1)|=\Omega(d^{4s+2}/n^s).
\end{equation}
In this sum, we may ignore terms that are too small to contribute meaningfully.
More precisely, let us fix a sufficiently large constant $K_0$ (to be determined later), and remove  from the
above sum all $s$-tuples $(u_1,\cdots,u_s)$
for which $|P_{u_1,\cdots,u_s}(X_1)| < K_0$. Note that the total number of helpful blue $s$-stars excluded by this pruning is at most
	$O(d^{2s})=o(d^{4s+2}/n^s)$. Let  $U$ denote the set of $s$-tuples in $X_1$ such that $|P_{u_1,\cdots,u_s}(X_1)|\geq K_0$. Then, summing over all
	$(u_1, \cdots, u_s)$ in $U$, we conclude that the total remains of order $\Omega(d^{4s+2}/n^s)$.

Now consider a fixed $s$-tuple $(u_1,\cdots,u_s)\in U$. Since $G$ is $\mathcal F_{1,s,t}$-free, there exists no helpful copy of $K_{s,t}$ in $J[X_1]$.
Thus, we can find a maximal subset $Y=\{y_1,\cdots,y_p\}\subseteq P_{u_1,\cdots,u_s}(X_1)$ with $p < t$, such that the colored complete bipartite graph
$K_{s,p}$ formed by the vertex sets $\{u_1,\cdots,u_s\}$ and $\{y_1,\cdots,y_p\}$, together with all blue edges between them,
 is helpful.
Consequently, for any vertex $y\in P_{u_1,\cdots,u_s}(X_1) \backslash Y$, there
must exist a common neighbor shared  with some pair $(u_i, y_j)$ from the sets above. By pigeonhole principle, there exists: 
\begin{enumerate}
\item a set $X\subset P_{u_1,\cdots,u_s}(X_1)$, of size $(|P_{u_1,\cdots,u_s}(X_1)|-p)/sp =\Omega(|P_{u_1,\cdots,u_s}(X_1)|)$
\item an index $k\in \{1,\ldots,s\}$, and
\item a vertex $b\in B$,  
\end{enumerate}
 such that
$X\cup \{u_k\} \subset N_G(b)$.

Note once again that the existence of a red $(s+t)$-clique in $X_1$ would contradict the assumption
 that $G$ is $\mathcal F_{1,s,t}$-free.
Therefore, we apply Proposition~\ref{Ramsey_stuff}
with $n_0=s+t$ and $A'=X$. It follows that the total number of non-red $s$-cliques in $X$ is at least $\Omega(|X|^s)$,
which is, in turn, at least
\begin{equation}
	\frac{1}{C_0} |P_{u_1,\cdots,u_s}(X_1)|^s,
\end{equation} for some constant $C_0>0$.

Now we choose the constant $K_0$ to be much larger than $C_0$.
Then, for each $s$-tuple $(u_1,\cdots,u_s) \in U$,
we can obtain at least $\frac{1}{C_0}|P_{u_1,\cdots,u_s}(X_1)|^s$ many $s$-tuples, each of which,
denoted by $(v_1,\cdots,v_s)$, forms a non-red clique in $J[X_1]$, which corresponds to Condition (1) in this proposition.
Note that each edge in this clique is either blue or uncolored; it is possible that the clique contains no blue edges at all.
Additionally, each $v_i$ is a blue neighbor of every $u_j$ for $j=1,\cdots, s$.
Furthermore, for the vertex $b\in B$ and the fixed $u_k$ with $1\leq k\leq s$, we have 
$$\{v_1,\cdots,v_s\} \cup \{u_k\} \subset N_G(b).$$
We now make the following estimate.
The total number of such $2s$-tuples \begin{equation}
\{u_1, \cdots, u_s, v_1,\cdots,v_s\}
\end{equation} is at least
\begin{align}
 \Omega \big( \sum_{U}|P_{u_1,\cdots,u_s}(X_1)|^s \big)
&=   \Omega \big( |U| \times \big(\frac 1{|U|}\sum_{U}|P_{u_1,\cdots,u_s}(X_1)| \big)^s \big) \\
& = \Omega \big( (d^{4s+2}/n^s)^s/(d^{2s})^{s-1})=
\Omega(d^{2s^2+4s}/n^{s^2} \big).\nonumber
\end{align}
Here we have used Jensen's inequality.

Next we perform a double counting.
First, we can choose a position $k$ in $s$ ways. Then, we can choose $b\in B$ in at most $n$ ways.
Next, we choose the values of the coordinates
$v_1,\cdots,v_s, u_k$ in $O(d^{s+1})$ ways,
since they are all neighbors of the chosen vertex $Y$. Therefore, the total number of such choices is $O(nd^{s+1})$ so far.

The degree condition  $d=\omega(n^{(s^2+1)/(2s+2s^2)})$
implies that
\begin{equation}\label{degree_double_counting}
d^{2s^2+4s}/n^{s^2} = \omega(d^{s-1} n d^{s+1}).
\end{equation}
Therefore, for some choice of $k, v_1,\cdots,v_s,u_k$ as described above,
there are at least  $\omega(d^{s-1})$ ways to choose the remaining  coordinates, resulting in
a $2s$-tuple satisfying all the conditions described as above. Since there are exactly $(s-1)$-coordinates left to be determined, we conclude that
there are at least $\omega(d)$ distinct vertices appearing in these coordinates.
Therefore, the conclusion of the proposition follows.
As a result, the proof of Theorem~\ref{family_version} is also completed,
by applying Proposition~\ref{End_of_Janzer}.
\end{proof}

\section*{Acknowledgment}
The authors would like to thank Robert Morris for carefully reviewing the manuscript, improving its readability, and providing valuable suggestions.

\bibliographystyle{abbrv}
\addcontentsline{toc}{chapter}{Bibliography}
\bibliography{sub_turan_r3}

\begin{thebibliography}{10}

\bibitem{alon2003turan}
N.~Alon, M.~Krivelevich, and B.~Sudakov.
\newblock Tur{\'a}n numbers of bipartite graphs and related ramsey-type
  questions.
\newblock {\em Combinatorics, Probability and Computing}, 12(5-6):477--494,
  2003.

\bibitem{bondy1974cycles}
J.~A. Bondy and M.~Simonovits.
\newblock Cycles of even length in graphs.
\newblock {\em Journal of Combinatorial Theory, Series B}, 16(2):97--105, 1974.

\bibitem{bukh2015random}
B.~Bukh.
\newblock Random algebraic construction of extremal graphs.
\newblock {\em Bulletin of the London Mathematical Society}, 47(6):939--945,
  2015.

\bibitem{bukh2021extremal}
B.~Bukh.
\newblock Extremal graphs without exponentially-small bicliques.
\newblock {\em arXiv preprint arXiv:2107.04167}, 2021.

\bibitem{bukh2018rational}
B.~Bukh and D.~Conlon.
\newblock Rational exponents in extremal graph theory.
\newblock {\em Journal of the European Mathematical Society}, 20(7):1747--1757,
  2018.

\bibitem{Conlon2019graphs}
D.~Conlon.
\newblock Graphs with few paths of prescribed length between any two vertices.
\newblock {\em Bulletin of the London Mathematical Society}, 51(6):1015--1021,
  2019.

\bibitem{conlon2021more}
D.~Conlon, O.~Janzer, and J.~Lee.
\newblock More on the extremal number of subdivisions.
\newblock {\em Combinatorica}, 41(4):1--30, 2021.

\bibitem{conlon2021extremal}
D.~Conlon and J.~Lee.
\newblock On the extremal number of subdivisions.
\newblock {\em International Mathematics Research Notices},
  2021(12):9122--9145, 2021.

\bibitem{erdHos1981combinatorial}
P.~Erd{\H o}s.
\newblock On the combinatorial problems which {I} would most like to see
  solved.
\newblock {\em Combinatorica}, 1(1):25--42, 1981.

\bibitem{erdous1988problems}
P.~Erd{\H o}s.
\newblock Problems and results in combinatorial analysis and graph theory.
\newblock {\em Discrete Mathematics}, 72(1--3):81--92, 1988.

\bibitem{erdos1946structure}
P.~Erd{\H o}s and A.~H. Stone.
\newblock On the structure of linear graphs.
\newblock {\em Bulletin of the American Mathematical Society},
  52(12):1087--1091, 1946.

\bibitem{Fredi1991OnAT}
Z.~F{\"u}redi.
\newblock On a tur{\'a}n type problem of erd{\"o}s.
\newblock {\em Combinatorica}, 11:75--79, 1991.

\bibitem{Furedi2013history}
Z.~F{\"u}redi and M.~Simonovits.
\newblock The history of degenerate (bipartite) extremal graph problems.
\newblock In {\em Erd{\H{o}}s Centennial}, pages 169--264. Springer, 2013.

\bibitem{janzer2018improved}
O.~Janzer.
\newblock Improved bounds for the extremal number of subdivisions.
\newblock {\em The Electronic Journal of Combinatorics}, 26(3), 2019.

\bibitem{janzer2020extremal}
O.~Janzer.
\newblock The extremal number of the subdivisions of the complete bipartite
  graph.
\newblock {\em SIAM Journal on Discrete Mathematics}, 34(1):241--250, 2020.

\bibitem{janzer2021extremal}
O.~Janzer.
\newblock The extremal number of longer subdivisions.
\newblock {\em Bulletin of the London Mathematical Society}, 53(1):108--118,
  2021.

\bibitem{jiang2020turan}
T.~Jiang and Y.~Qiu.
\newblock Tur{\'a}n numbers of bipartite subdivisions.
\newblock {\em SIAM Journal on Discrete Mathematics}, 34(1):556--570, 2020.

\bibitem{jiang2023many}
T.~Jiang and Y.~Qiu.
\newblock Many tur{\'a}n exponents via subdivisions.
\newblock {\em Combinatorics, Probability and Computing}, 32(1):134--150, 2023.

\bibitem{kang2021rational}
D.~Y. Kang, J.~Kim, and H.~Liu.
\newblock On the rational {T}ur{\'a}n exponents conjecture.
\newblock {\em Journal of Combinatorial Theory, Series B}, 148:149--172, 2021.

\bibitem{kollar1996norm}
J.~Koll{\'a}r, L.~R{\'o}nyai, and T.~Szab{\'o}.
\newblock Norm-graphs and bipartite tur{\'a}n numbers.
\newblock {\em Combinatorica}, 16(3):399--406, 1996.

\bibitem{kHovari1954problem}
P.~K{\H{o}}v{\'a}ri, V.~T~S{\'o}s, and P.~Tur{\'a}n.
\newblock On a problem of zarankiewicz.
\newblock In {\em Colloquium Mathematicum}, volume~3, pages 50--57. Polska
  Akademia Nauk, 1954.

\bibitem{sudakov2020turan}
B.~Sudakov and I.~Tomon.
\newblock Tur{\'a}n number of bipartite graphs with no $k_{\{t, t\}}$.
\newblock {\em Proceedings of the American Mathematical Society},
  148(7):2811--2818, 2020.

\end{thebibliography}

\end{document}